\documentclass[12pt,a4paper]{article}
\usepackage[utf8]{inputenc}
\usepackage[T1]{fontenc}
\usepackage{hyperref}
\usepackage{amsmath}
\usepackage{amsfonts}
\usepackage{amssymb}
\usepackage{xcolor}
\usepackage{graphicx}
\providecommand{\U}[1]{\protect\rule{.1in}{.1in}}

\newtheorem{theorem}{Theorem}

\newtheorem{conjecture}[theorem]{Conjecture}
\newtheorem{corollary}[theorem]{Corollary}

\newtheorem{lemma}[theorem]{Lemma}

\newtheorem{proposition}[theorem]{Proposition}
\newtheorem{remark}[theorem]{Remark}

\newenvironment{proof}[1][Proof]{\noindent\textbf{#1.} }{\ \hfill \rule{0.5em}{0.5em}\bigskip}

\graphicspath{{Slike/}}

\begin{document}

\title{Resolutions of two conjectures on the spectral diameter}

\author{Irena M. Jovanovi\'c\\
\footnotesize{\textbf{former} School of Computing, Union University, Serbia}\\
\footnotesize{{\it Email:} irenaire@gmail.com}\vspace{0.5em}\\
Jelena Sedlar \\
\footnotesize{Faculty of civil engineering, architecture and geodesy, University of Split, Croatia}\\
\footnotesize{{\it Email:} jsedlar@gradst.hr}\vspace{0.5em}\\
Riste Škrekovski\\
\footnotesize{Faculty of Mathematics and Physics, University of Ljubljana, Ljubljana, Slovenia;}\\
\footnotesize{Rudolfovo -- Science and Technology Centre Novo Mesto, Slovenia;}\\
\footnotesize{Faculty of Information Studies, University in Novo Mesto, Slovenia}\\
\footnotesize{{\it Email:} skrekovski@gmail.com}}

\date{}
\maketitle

\begin{abstract}
Let $\lambda_1(G)\geq\cdots\geq\lambda_n(G)$ be the adjacency spectrum of a graph $G$ on $n$ vertices. The spectral distance $\sigma(G,H)$ between $n$-vertex graphs $G$ and $H$ is the Manhattan distance between their spectra, i.e. $\sigma(G,H)=\sum_{i=1}^{n} |\lambda_i(G)-\lambda_i(H)|$. Given a set $\mathcal{G}$ of pairwise non-isomorphic graphs of order $n$, the spectral diameter of $\mathcal{G}$ is defined as $\mathrm{sdiam}(\mathcal{G})=\max\{\mathrm{secc}_{\mathcal{G}}(G):G\in\mathcal{G}\}$, where $\mathrm{secc}_{\mathcal{G}}(G)=\max\{\sigma(G,H):H\in\mathcal{G}\}$ is the spectral eccentricity of $G\in\mathcal{G}$. Among six conjectures on spectral distances posed by Z. Stani\'c in 2012, two conjectures related to the spectral diameter of certain graph classes remained open. One of them concerns the spectral diameter of the set $\mathcal{B}_n$ of all connected bipartite graphs of order $n$, while the other, of the set $\mathcal{T}_n$ of all trees of order $n$. More precisely, Stani\'c conjectured that ${\rm sdiam}(\mathcal{T}_n)=\sigma(P_n, K_{1,n-1})$, where $P_n$ is the path graph, while $K_{1,n-1}$ is the star, and that ${\rm sdiam}(\mathcal{B}_n)={\rm secc}_{\mathcal{B}_n}(K_{\lceil\frac{n}{2}\rceil, \lfloor\frac{n}{2}\rfloor})$, where $K_{\lceil\frac{n}{2}\rceil, \lfloor\frac{n}{2}\rfloor}$ is the complete bipartite graph. In this paper, both of these conjectures are disproved.
\end{abstract}



\section{Introduction}

Let $G$ be a finite, simple and undirected graph on $n$ vertices and $m$ edges, and let $A=A(G)=[a_{ij}]$, $i,j=1,2,\ldots,n$, be the \emph{adjacency matrix} of $G$. The \emph{cha\-ra\-cte\-ri\-stic polynomial} $P_G(x)=\det(A-xI_n)$ of $G$, where $I_n$ is the identity matrix, is the characteristic polynomial of its adjacency matrix $A$, while the \emph{(adjacency) eigenvalues} $\lambda_1(G)\geq\lambda_2(G)\geq\cdots\geq\lambda_n(G)$ of $G$ are the eigenvalues of its adjacency matrix $A$. We will always assume that the eigenvalues of graphs considered throughout the paper are ordered monotonically decreasing. The listed eigenvalues form the \emph{spectrum} of $G$. The largest eigenvalue $\lambda_1(G)$ is called the \emph{index} of $G$, while if $\lambda_i(G)$, for some $i$, is the eigenvalue of the multiplicity $k$, we will denote it as $[\lambda_i(G)]^k$. For an integer $r\geq 0$, the \emph{$r$-th spectral moment} $s_r$ of $G$ is defined as $s_r=\sum_{i=1}^{n} \lambda_i(G)^r$. It holds \cite{CvRowSim} that $s_1=0$ and $s_2=2m$.

A graph $G$ is called \emph{bipartite} if its vertex set $V(G)$ can be partitioned into two subsets, say $X$ and $Y$, so that each edge has one vertex in $X$ and one in $Y$; such a partition $(X,Y)$ is called a \emph{bipartition} of the graph. The \emph{complete bipartite graph} $K_{p,q}$ is a bipartite graph with the bipartition $(X, Y)$, where $|X|=p$ and $|Y|=q$, or vice versa, in which each vertex of $X$ is adjacent to each vertex of $Y$. A complete bipartite graph where one of the sets of the bipartition is of the cardinality $1$ is called the \emph{star}. A bipartite graph is said to be \emph{balanced} if $|X|=|Y|$, and \emph{unbalanced} otherwise. It is well known (Theorem 3.2.3 in \cite{CvRowSim}) that a graph $G$ is bipartite if and only if its spectrum is symmetric with respect to zero, i.e. $\lambda_{n+1-i}(G)=-\lambda_{i}(G)$, for every $1\leq i\leq n$. The adjacency matrix of a bipartite graph $G$ with the bipartition $(X,Y)$ is of the form $A=A(G)=\left(
                                                                                                                   \begin{array}{cc}
                                                                                                                     0 & B \\
                                                                                                                     B^T & 0 \\
                                                                                                                   \end{array}
                                                                                                                 \right),
$ where $B\in\{0,1\}^{|X|\times |Y|}$ is called the \emph{biadjacency matrix}. Therefore, the eigenvalues of $A$, i.e. $G$, equal $\pm\sigma_1,\ldots,\pm\sigma_l$, together with a suitable number of zeros, where $l=\min\{|X|,|Y|\}$, while $\sigma_1\geq\cdots\geq\sigma_l$ are the singular values of $B$. Here, the singular values of $B$ are the square roots of nonnegative eigenvalues of $B^TB$. For the remaining notation and terminology about graphs and their spectra we refer the reader to \cite{BrouHaem}, \cite{CvDoobSachs} and \cite{CvRowSim}.

Let $G_1$ and $G_2$ be two non-isomorphic graphs on $n$ vertices with the adjacency spectra $\lambda_1(G_i)\geq\lambda_2(G_i)\geq\cdots\geq\lambda_n(G_i)$, $i=1,2$. The \emph{spectral distance} $\sigma(G_1,G_2)$ between $G_1$ and $G_2$ is \cite{aveiro}
\begin{equation*}
\sigma(G_1,G_2)=\sum\limits_{i=1}^{n} |\lambda_i(G_1)-\lambda_i(G_2)|.
\end{equation*}
The study of such measure, i.e. quantity originates from a group of problems proposed by Richard Brualdi at the Aveiro Workshop on Graph Spectra in 2006 (see \cite{aveiro}), and it has been pursued ever since by several authors \cite{AbdJanObo, AbdObo, AbdZakArX, AbdZak, az:conjecture, Jov2022, Jov2015, JovStan2014, JovStan2012, Obo}. Brualdi's problems, which were originally connected with the \emph{cospectrality} of graphs and the \emph{cospectrality measure} of certain sets of non-isomorphic graphs of the same order, were, in certain sense, upgraded in the paper \cite{JovStan2012} by introducing the two additional spectral distance related parameters named spectral eccentricity and spectral diameter. 

Given a set $\mathcal{G}$ of pairwise non-isomorphic graphs of order $n$, the \emph{spectral eccentricity} of $G\in\mathcal{G}$ and the \emph{spectral diameter} of
$\mathcal{G}$ are defined as follows
\begin{equation*}
\mathrm{secc}_{\mathcal{G}}(G)=\max\{\sigma(G,H):H\in\mathcal{G}\}\quad\text{and}\quad
\mathrm{sdiam}(\mathcal{G})=\max\{\mathrm{secc}_{\mathcal{G}}(G):G\in\mathcal{G}\},
\end{equation*}
respectively. In \cite{JovStan2012}, Z. Stani\'c posed six conjectures regarding spectral distances of graphs of which four relate to the spectral diameter of some well known and widely studied graph classes. Two of these four conjectures, which concern the diameter of the set of all graphs of order $n$ and the diameter of the set of all connected regular graphs of order $n$, were disproved in \cite{Jov2022, Jov2015}. Here, we consider the remaining two.

\begin{conjecture}\label{con_trees}
Let $\mathcal{T}_n$ be the set of all trees of order $n$. The spectral distance between any two trees from the set $\mathcal{T}_n$ does not exceed the spectral distance $\sigma(P_n, K_{1,n-1})$ between the path $P_n$ and the star graph $K_{1,n-1}$ of order $n$, i.e.
$${\rm sdiam}(\mathcal{T}_n)=\sigma(P_n, K_{1,n-1}).$$
\end{conjecture}

\begin{conjecture}\label{con_bip}
Let $\mathcal{B}_n$ be the set of all connected bipartite graphs of order $n$, and let $B_1, B_2\in \mathcal{B}_n$ be the graphs having the maximal spectral distance on the given set. Then one of them is the complete bipartite graph $K_{\lceil\frac{n}{2}\rceil, \lfloor\frac{n}{2}\rfloor}$, i.e.
$${\rm sdiam}(\mathcal{B}_n)={\rm secc}_{\mathcal{B}_n}(K_{\lceil\frac{n}{2}\rceil, \lfloor\frac{n}{2}\rfloor}).$$
\end{conjecture}

Recall that the \emph{energy} $\mathcal{E}(G)$ of a graph $G$ is the sum of the absolute values of its eigenvalues, i.e. $\mathcal{E}(G)=\sum_{i=1}^{n} |\lambda_i(G)|$. This widely studied graph invariant was introduced by Ivan Gutman in 1978 in his paper \cite{Gutman}, and for the last two decades it has been intensively studied, especially in the field of mathematics and chemistry. For more details about graph energy, the reader is referred to the monographs \cite{GutFurtula2}, \cite{GutLi} and \cite{Li}, and the review paper \cite{GutFurtula}.

It can be easily verified that $\sigma(G,H)\leq\mathcal{E}(G)+\mathcal{E}(H)$ holds for every pair of graphs $G$ and $H$ of the same order. Therefore, it seems graphs of extremal energy can be considered as candidates for pairs of graphs attaining the value of the spectral diameter on the sets of graphs under study, and this is what the conjectures of Stani\'{c} on the spectral diameter of some well known graph classes can be understood. However, Conjectures \ref{con_trees} and \ref{con_bip} both fail. In Section 2, we disprove Conjecture \ref{con_trees} by considering the spectral distance between the path graph and a tree which is the coalescence of two stars, while in Section 3, we disprove Conjecture \ref{con_bip} by considering the spectral distance between the incidence graph of a Menon symmetric design and a bipartite chain graph.

\section{The spectral diameter of trees}

In the literature related to some spectral characterizations of graphs, the graph operation named coalescence is frequently used. The \emph{coalescence} $G\cdot H$ of two arbitrary graphs $G$ and $H$ is a graph obtained from their (disjoint) union $G\cup H$ by identifying a vertex $u$ of $G$ with a vertex $v$ of $H$. The proof of the following statement can be found in \cite[Theorem 2.2.3]{CvRowSim}.

\begin{theorem}\label{coalescence}
Let $G\cdot H$ be the coalescence in which the vertex $u$ of $G$ is identified with the vertex $v$ of $H$. Then
\begin{equation*}
P_{G\cdot H}(x)=P_G(x)\,P_{H-v}(x)+P_{G-u}(x)\,P_H(x)-x\,P_{G-u}(x)\,P_{H-v}(x),
\end{equation*}
where $G-u$ denotes the graph obtained by removing the vertex $u$ from the graph $G$ (and similarly for $H$ and its vertex $v$).
\end{theorem} 

Let $T_{a,b}$ be the coalescence in which a pendant vertex (i.e. vertex whose degree is equal $1$) of the star $K_{1,a+1}$ is identified with a pendant vertex of another star $K_{1,b+1}$, where $a,b\geq 1$. Obviously, $T_{a,b}$ is a tree of order $a+b+3$. The tree $T_{a,b}$ for $a=4$ and $b=5$ is depicted in Figure \ref{treeTabFig}. We say that $T_{a,b}$ is \emph{balanced} if $a=\left\lfloor\frac{a+b}{2}\right\rfloor$ and $b=\left\lceil\frac{a+b}{2}\right\rceil$, i.e. $a=b$, in which case the order of $T_{a,b}$ is an odd number, or $b=a+1$, in which case the order of $T_{a,b}$ is an even number. The spectrum of $T_{a,b}$ is given by the following statement.

\begin{figure}[h]
%\vspace{-10pt}
\centering
\includegraphics[scale=1]{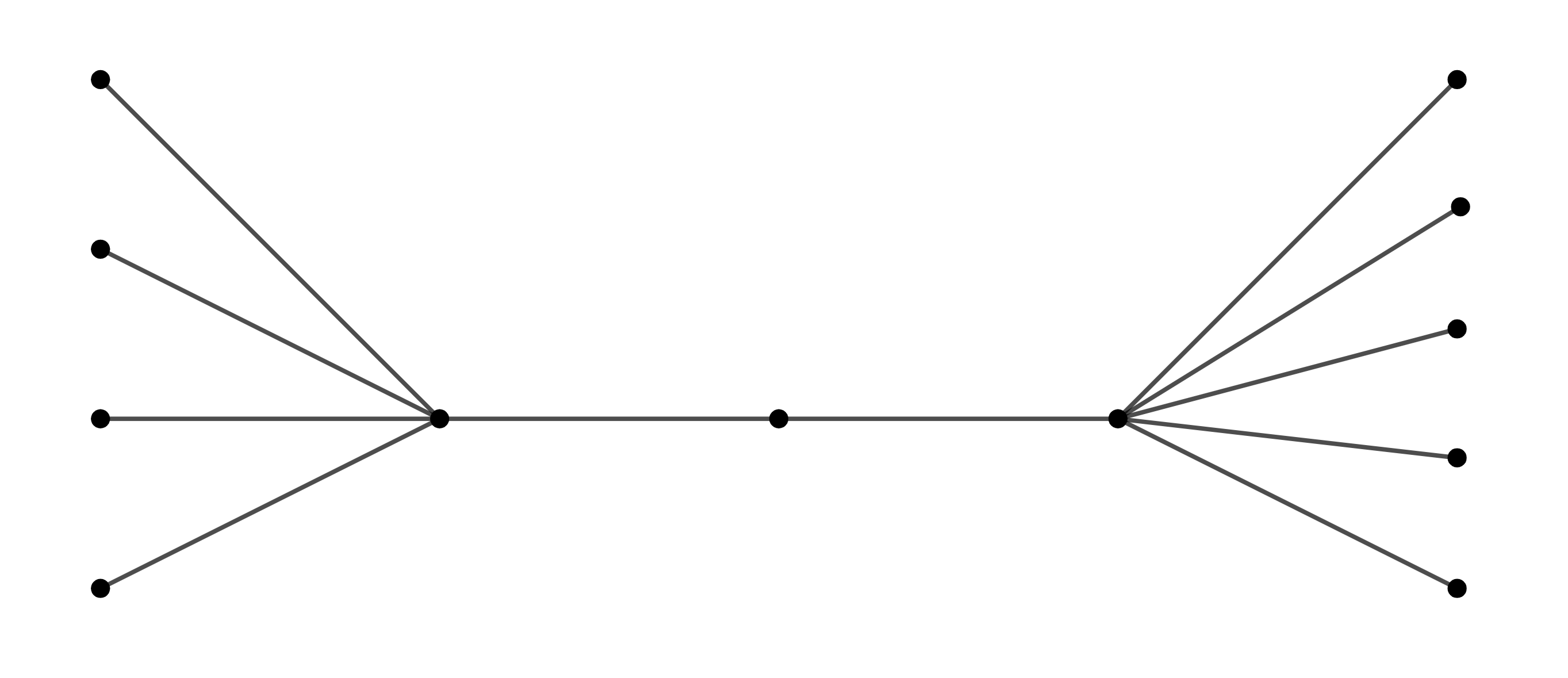}
\vspace{10pt}
\caption{Tree $T_{4,5}$}
\label{treeTabFig}
\end{figure}
%\vspace{15pt}

\begin{lemma}\label{spectrum}
Let $a,b\geq 1$. The spectrum of the $n$-vertex tree $T_{a,b}$ has exactly four non-zero eigenvalues $\lambda_1(T_{a,b})=\sqrt{\frac{a+b+2+\sqrt{(a-b)^2+4}}{2}}$, $\lambda_2(T_{a,b})=\sqrt{\frac{a+b+2-\sqrt{(a-b)^2+4}}{2}}$, $\lambda_{n-1}(T_{a,b})=-\lambda_2(T_{a,b})$ and $\lambda_{n}(T_{a,b})=-\lambda_1(T_{a,b})$. Furthermore, the two largest eigenvalues satisfy the following identity 
\begin{equation}\label{eigen_sum}
\lambda_{1}(T_{a,b})+\lambda_{2}(T_{a,b})=\sqrt{a+b+2+2\sqrt{ab+a+b}}.
\end{equation} 
\end{lemma}

\begin{proof}
Since the characteristic polynomial of the star $K_{1,\alpha}$ is (see Section 2.6 in \cite{CvRowSim}) $P_{K_{1,\alpha}}(x)=(x^2-\alpha)\,x^{\alpha-1}$, and since by removing a pendant vertex of a star whose order is $\alpha+1$, we obtain the star of order $\alpha$, by use of Theorem \ref{coalescence}, we easily find the characteristic polynomial of the tree $T_{a,b}$
\begin{align*}
P_{T_{a,b}}(x) = &\, P_{K_{1,a+1}}(x)\cdot P_{K_{1,b}}(x)+P_{K_{1,a}}(x)\cdot P_{K_{1,b+1}}(x)-x\cdot P_{K_{1,a}}(x)\cdot P_{K_{1,b}}(x)\\
               = &\, (x^2-a-1)(x^2-b)\,x^{a+b-1}+(x^2-a)(x^2-b-1)\,x^{a+b-1}\\
                 &\, - (x^2-a)(x^2-b)\,x^{a+b-1},
\end{align*}
that is
\begin{equation*}\label{t_ab_char}
P_{T_{a,b}}(x)=x^{a+b-1}\,(x^4-(a+b+2)x^2+ab+b+a).
\end{equation*}
So, the spectrum of $T_{a,b}$ consists of the following eigenvalues $[0]^{a+b-1}$, $\pm \sqrt{\frac{a+b+2+\sqrt{(a-b)^2+4}}{2}}$ and $\pm\sqrt{\frac{a+b+2-\sqrt{(a-b)^2+4}}{2}}$. It is obvious that $\lambda_1(T_{a,b})=\sqrt{\frac{a+b+2+\sqrt{(a-b)^2+4}}{2}}$ and $\lambda_2(T_{a,b})=\sqrt{\frac{a+b+2-\sqrt{(a-b)^2+4}}{2}}$, as well as $\lambda_{n-1}(T_{a,b})=-\lambda_2(T_{a,b})$ and $\lambda_{n}(T_{a,b})=-\lambda_1(T_{a,b})$, since $T_{a,b}$ is bipartite. This proves the first part of the statement.

For the second part, let us notice that by direct computation one finds
$$\lambda_{1}(T_{a,b})^{2}+\lambda_{2}(T_{a,b})^{2}=a+b+2,$$ 
while from Vieta's formulas it follows 
$$\lambda_{1}(T_{a,b})\,\lambda_{2}(T_{a,b})=\sqrt{ab+a+b}.$$ 
Therefore, using the previous two equalities in the formula for the square of the binomial consisting of $\lambda_1(T_{a,b})$ and $\lambda_2(T_{a,b})$, we obtain (\ref{eigen_sum}).
\end{proof}

\begin{lemma}\label{balanced_tree}
Let $T_{a,b}$, $a,b\geq 1$, be balanced. Then $\lambda_2(T_{a,b})>2$, for every integer $a>4$. 
\end{lemma}

\begin{proof}
We should prove that $\lambda_2(T_{a,b})=\sqrt{\frac{a+b+2-\sqrt{(a-b)^2+4}}{2}}>2$, which is equivalent to $a+b-6>\sqrt{(a-b)^2+4}$. Since $a>4$ and $T_{a,b}$ is balanced, $a+b>6$, so the last inequality becomes 
\begin{equation}\label{inequality}
8+ab-3a-3b>0.
\end{equation}
Given that $T_{a,b}$ is balanced, $b=a$ or $b=a+1$. If $b=a$, the inequality (\ref{inequality}) reduces to $a^2-6a+8>0$, which holds for $a>4$. If $b=a+1$, the inequality (\ref{inequality}) is of the form $a^2-5a+5>0$, which is satisfied for every $a\geq4$. This completes the proof. 
\end{proof}

In order to prove the main statement of the section, we need the following lemma.

\begin{lemma}\label{function}
Let $F(x)=\sqrt{x+\sqrt{x^2-5}}-\sqrt{x}$, where $x\geq 3$. Then, $F(x)>4$, for every $x\geq 94$.
\end{lemma}

\begin{proof}
Two successive squarings show that the inequality $F(x)>4$ is equivalent first to $\sqrt{x^2-5}>16+8\,\sqrt{x}$, and then to $x^2-64x-256\sqrt{x}-261>0$. Let us denote by $G(x)=x^2-64x-256\sqrt{x}-261$. Since $G(93)\approx-32.77<0$ and $G(94)\approx 76.99>0$, the function $G(x)$ has a zero in the interval $(93,94)$. Now, we have $G^{\prime}(x)=2x-64-\frac{128}{\sqrt{x}}$, wherefrom we find that for $x\geq 50$, $G^{\prime}(x)\geq 2\cdot 50-64-\frac{128}{\sqrt{50}}\approx 17.90>0$. This means that $G(x)$ is monotonically increasing on the interval $[50,+\infty)$, and having in mind that $G(94)>0$, we conclude that $G(x)>0$ in the interval $[94,+\infty)$. The proof is completed.
\end{proof}

Now we are ready to consider the main claim of this section.

\begin{theorem}\label{main1}
For every $n\geq 95$, it holds 
$$\sigma(P_n, T_{a,b})>\sigma(P_n,K_{1,n-1}),$$
where $T_{a,b}$ is balanced with $a+b=n-3$.
\end{theorem}

\begin{proof}
According to the results exposed in Table 2 in \cite{JovStan2012}, for every $n\geq 5$, it holds
\begin{equation}\label{distancePK}
\sigma(P_n,K_{1,n-1})=\mathcal{E}(P_n)+2\sqrt{n-1}-4\lambda_1(P_n).
\end{equation}

Let us now compute the spectral distance between $P_n$ and $T_{a,b}$. Having in mind that $P_n$ and $T_{a,b}$ are bipartite graphs, and that, according to Lemma \ref{spectrum}, $T_{a,b}$ has four non-zero eigenvalues, we have
\begin{equation}\label{auxiliary1}
\sigma(P_n,T_{a,b})=2\,\sum\limits_{i=1}^{2} |\lambda_i(P_n)-\lambda_i(T_{a,b})|+\sum\limits_{i=3}^{n-2} |\lambda_i(P_n)|.
\end{equation} 
Since the eigenvalues of $P_n$ are (see Section 2.6 in \cite{CvRowSim}) $2\,\cos\frac{i\pi}{n+1}$, for $i=1,2,\ldots,n$, it holds $\lambda_2(P_n)\leq\lambda_1(P_n)<2$. According to Lemma \ref{balanced_tree}, we have $\lambda_1(T_{a,b})\geq \lambda_2(T_{a,b})>2$. Therefore, (\ref{auxiliary1}) becomes
\begin{equation*}
\sigma(P_n,T_{a,b})=\mathcal{E}(P_n)+2\,(\lambda_1(T_{a,b})+\lambda_2(T_{a,b}))-4\,(\lambda_1(P_n)+\lambda_2(P_n)),
\end{equation*}
that is, taking into account (\ref{eigen_sum}),
\begin{equation}\label{distancePT}
\sigma(P_n,T_{a,b})=\mathcal{E}(P_n)+2\,\sqrt{a+b+2+2\sqrt{ab+a+b}}-4\,(\lambda_1(P_n)+\lambda_2(P_n)).
\end{equation}
Now, from (\ref{distancePK}) and (\ref{distancePT}), we obtain
\begin{equation}\label{*}
\sigma(P_n,T_{a,b})-\sigma(P_n,K_{1,n-1})=2\,\sqrt{a+b+2+2\sqrt{ab+a+b}}-2\,\sqrt{a+b+2}-4\lambda_2(P_n).
\end{equation}
In order to prove the statement, we will show that $\sigma(P_n,T_{a,b})-\sigma(P_n,K_{1,n-1})>0.$ Let us first show that
$$\sqrt{a+b+2+2\sqrt{ab+a+b}}-\sqrt{a+b+2}>4.$$
According to the assumption, $T_{a,b}$ is balanced, which means that $b=a$ or $b=a+1$. Therefore, 
$$4(ab+a+b)=\left\{
              \begin{array}{ll}
                (a+b+2)^2-4, & \hbox{for\, $b=a$;} \\
                (a+b+2)^2-5, & \hbox{for\, $b=a+1$,}
              \end{array}
            \right.
$$
that is, 
$$2\sqrt{ab+a+b}=\left\{
                   \begin{array}{ll}
                     \sqrt{(a+b+2)^2-4}, & \hbox{for\, $b=a$;} \\
                     \sqrt{(a+b+2)^2-5}, & \hbox{for\, $b=a+1$.}
                   \end{array}
                 \right.
$$
From the last equalities, we can conclude that $2\sqrt{ab+a+b}\geq\sqrt{(a+b+2)^2-5}$. Hence, we have
\begin{align*}
\sqrt{a+b+2+2\sqrt{ab+a+b}}-\sqrt{a+b+2} & \geq \\
\sqrt{a+b+2+\sqrt{(a+b+2)^2-5}}-\sqrt{a+b+2} & =F(a+b+2),
\end{align*}
where $F$ is the function defined as in Lemma \ref{function}. Next, by use of Lemma \ref{function}, it follows that 
$$\sqrt{a+b+2+2\sqrt{ab+a+b}}-\sqrt{a+b+2}>4,$$
for $a+b\geq 92$, i.e. $n\geq 95$. Now, (\ref{*}) reduces to
$$\sigma(P_n,T_{a,b})-\sigma(P_n,K_{1,n-1})>2\cdot 4-4\lambda_2(P_n)\geq 8-8=0,$$
since $\lambda_2(P_n)<2$, which completes the proof.
\end{proof}

\noindent From Theorem \ref{main1}, it follows ${\rm sdiam}(\mathcal{T}_n)\geq\sigma(P_n,T_{a,b})>\sigma(P_n, K_{1,n-1})$, which means that Conjecture \ref{con_trees} is false. 

Although Theorem \ref{main1} is proven for $n\geq 95$, certain computational results show that a counterexample to Conjecture \ref{con_trees} exists for $n=94$. In that case, the difference between the values of the spectral distances $\sigma(P_n,T_{a,b})$ and $\sigma(P_n, K_{1,n-1})$ is approximately $0.00458$. We also verified by an  exhaustive computer search that Conjecture \ref{con_trees} holds for all $n\leq 18$.

\section{The spectral diameter of bipartite graphs}

Throughout this section, we consider connected bipartite graphs of even order $n$, and, in particular, bipartite graph $K_{\frac{n}{2},\frac{n}{2}}$, whose spectrum consists of the following eigenvalues (see, for exa\-mple, Section 2.6 in \cite{CvRowSim}): $\frac{n}{2}$, $[0]^{n-2}$, and $-\frac{n}{2}$. It is very well known that $\lambda_1(K_{\frac{n}{2},\frac{n}{2}})=\frac{n}{2}\geq\lambda_1(G)$, for every other $n$-vertex bipartite graph $G$. 

In order to prove the main result of the section, we need the following auxiliary statements.

\begin{lemma}\label{bipdistance}
The spectral distance between bipartite graphs $G_1$ and $G_2$ of even order $n$ equals
\begin{equation*}
\sigma(G_1,G_2)=\mathcal{E}(G_1)+\mathcal{E}(G_2)-4\,\sum\limits_{i=1}^{n/2}\min\{\lambda_i(G_1), \lambda_i(G_2)\}.
\end{equation*}
\end{lemma}

\begin{proof}
Since the spectra of $G_1$ and $G_2$ are symmetric with respect to the origin, we derive
\begin{align*}
\sigma(G_1,G_2) = &\, 2\,\sum\limits_{i=1}^{n/2}|\lambda_i(G_1)-\lambda_i(G_2)|\\
                = &\, 2\,\sum\limits_{i=1}^{n/2}(\lambda_i(G_1)+\lambda_i(G_2)-2\min\{\lambda_i(G_1), \lambda_i(G_2)\})\\
                = &\, \mathcal{E}(G_1)+\mathcal{E}(G_2)-4\,\sum\limits_{i=1}^{n/2}\min\{\lambda_i(G_1), \lambda_i(G_2)\}.
\end{align*}
\end{proof}

%\begin{remark}
%The statement given by Lemma \ref{bipdistance} holds in the case of bipartite graphs of odd order, as well.
%\end{remark}

\begin{remark}
By Lemma \ref{bipdistance}, any bipartite graph $G$ of even order $n$ satisfies
\begin{equation}\label{sdistmenon}
\sigma(K_{\frac{n}{2},\frac{n}{2}},G)=\mathcal{E}(G)+n-4\lambda_1(G).
\end{equation}
Hence, on the set $\mathcal{B}_n$ of all connected bipartite graphs of even order $n$ it holds
\begin{equation}\label{remark_secc}
{\rm secc}_{\mathcal{B}_n}\left(K_{\frac{n}{2},\frac{n}{2}}\right)=n+\max\limits_{G\in \mathcal{B}_n}(\mathcal{E}(G)-4\lambda_1(G)).
\end{equation}
\end{remark}

\begin{theorem}\label{seccB_n}
Let $\mathcal{B}_n$ be the set of all connected bipartite graphs of even order $n$. Then 
\begin{equation}\label{bound}
{\rm secc}_{\mathcal{B}_n}\left(K_{\frac{n}{2},\frac{n}{2}}\right)\leq\left(\frac{n}{2}\right)^{3/2}+\frac{n}{2}.
\end{equation}
\end{theorem}

\begin{proof}
Let $G\in \mathcal{B}_n$ be an arbitrary bipartite graph with the bipartition $(X,Y)$, where $|X|=a$ and $|Y|=b$, $a,b\geq 1$, and let $l=\min\{a,b\}$. In order to estimate the value of ${\rm secc}_{\mathcal{B}_n}\left(K_{\frac{n}{2},\frac{n}{2}}\right)$, based on (\ref{remark_secc}), we should estimate the quantity $\mathcal{E}(G)-4\lambda_1(G)$ on the set $\mathcal{B}_n$.  

Let us denote by $B$ the biadjacency matrix of $G$ whose dimension is $a\times b$, while by $m$ the number of its edges. As we have already noticed in the introductory part, the non-zero eigenvalues $\pm\sigma_1,\ldots,$ $\pm\sigma_l$ of $G$ are the singular values $\sigma_1\geq\cdots\geq\sigma_l$ of its biadjacency matrix $B$, and it holds $\sum\limits_{i}\sigma_i^2=m$. For a non-zero vector $x\in\mathbb{R}^b$, we have \cite{norm}
$$\sigma_1\geq\frac{\|Bx\|_2}{\|x\|_2},$$
where $\|\cdot\|_2$ is the  Euclidean norm. For all-1 vector $\textbf{j}\in\mathbb{R}^b$, we further find
$$\frac{\|B\,\textbf{j}\|_2}{\|\textbf{j}\|_2}=\frac{\sqrt{r_1^2+r_2^2+\cdots+r_a^2}}{\sqrt{b}},$$
where $r_i$, for $1\leq i\leq a$, denote the $i$-th row sum of the matrix $B$. By applying the Cauchy–Schwarz inequality from the last equality we finally get
\begin{equation}\label{rel1}
\sigma_1\geq \frac{m}{\sqrt{ab}}.
\end{equation} 
Bearing in mind that $\sum\limits_{i=1}^{l} \sigma_i^2=m$, and by applying the Cauchy–Schwarz inequality to $\sigma_2,\ldots,\sigma_l$, we obtain
\begin{equation}\label{rel2}
\mathcal{E}(G)-4\lambda_1(G)=2\,\sum\limits_{i=1}^{l} \sigma_i-4\sigma_1\leq2\,\sqrt{(m-\sigma_1^2)(l-1)}-2\sigma_1.
\end{equation}
Using (\ref{rel1}), i.e. $m\leq\sigma_1\cdot\sqrt{ab}$, (\ref{rel2}) reduces to
\begin{equation}\label{rel3}
\mathcal{E}(G)-4\lambda_1(G)\leq2\,\sqrt{\sigma_1(\sqrt{ab}-\sigma_1)(l-1)}-2\sigma_1.
\end{equation}
Since $ab\leq\frac{n^2}{4}$ and $l=\min\{a,b\}\leq\frac{n}{2}$, (\ref{rel3}) further reduces to
\begin{equation*}
\mathcal{E}(G)-4\lambda_1(G)\leq2\,\sqrt{\sigma_1\left(\frac{n}{2}-\sigma_1\right)\left(\frac{n}{2}-1\right)}-2\sigma_1. 
\end{equation*}
Let us denote by $F(\sigma_1)=2\,\sqrt{\sigma_1\left(\frac{n}{2}-\sigma_1\right)\left(\frac{n}{2}-1\right)}-2\sigma_1$, for $\sigma_1\in\left[0,\frac{n}{2}\right]$, and let us determine the maximum value of the function $F$ in the interval $\left[0,\frac{n}{2}\right]$.

Since $F^{\prime}(\sigma_1)=\frac{\left(\frac{n}{2}-1\right)\left(\frac{n}{2}-2\sigma_1\right)}{\sqrt{\sigma_1\left(\frac{n}{2}-\sigma_1\right)\left(\frac{n}{2}-1\right)}}-2$, the function $F$ has one stationary point $\frac{n-\sqrt{2n}}{4}$ in the interval $\left[0,\frac{n}{2}\right]$. It holds $F(0)=0$, $F(\frac{n}{2})=-n$ and $F\left(\frac{n-\sqrt{2n}}{4}\right)=\frac{n}{2}\left(\sqrt{\frac{n}{2}}-1\right)$, wherefrom we conclude that $F$ in the interval $\left[0,\frac{n}{2}\right]$ attains the maximum in $\frac{n-\sqrt{2n}}{4}$, and that it is equal to $\left(\frac{n}{2}\right)^{3/2}-\frac{n}{2}$.

Now, since $\mathcal{E}(G)-4\lambda_1(G)\leq\left(\frac{n}{2}\right)^{3/2}-\frac{n}{2}$, the proof follows from (\ref{remark_secc}).
\end{proof}

In the following, we will determine a graph $G$ such that $\sigma\left(K_{\frac{n}{2},\frac{n}{2}},G\right)=\left(\frac{n}{2}\right)^{3/2}+\frac{n}{2}$, i.e. $\sigma\left(K_{\frac{n}{2},\frac{n}{2}},G\right)={\rm secc}_{\mathcal{B}_n}\left(K_{\frac{n}{2},\frac{n}{2}}\right)$, where $\mathcal{B}_n$, as before, is the set of all connected bipartite graphs of even order $n$. To do this, we need to recall some definitions.

A \emph{design} $\mathcal{D}$ consists of a set of $v$ points and $b$ subsets of the set of these points called \emph{blocks} such that there are $k$ points per block, $r$ blocks per point and $\lambda^{\ast}$ blocks through any two distinct points. The integers $(v,b,k,r,\lambda^{\ast})$ are \emph{parameters} of the design, and they satisfy: $vr=bk$, $\lambda^{\ast}(v-1)=r(k-1)$ and $b\geq v$. If $b=v$, or equivalently, if $r=k$, the design $\mathcal{D}$ is \emph{symmetric}. A \emph{$t$-design} is a design $\mathcal{D}$ such that each set of $t$ points is in the same number $\lambda^{\ast}_t>0$ of blocks. 

Given a design $\mathcal{D}$, the incidence graph $\Gamma(\mathcal{D})$ of the design is formed as follows \cite{CvDoobSachs}: vertices of $\Gamma(\mathcal{D})$ correspond to the points and blocks of the design, so $\Gamma(\mathcal{D})$ is of order $b+v$; two vertices in $\Gamma(\mathcal{D})$ are adjacent if and only if one corresponds to a block and the other corresponds to a point contained in that block. It is clear that $\Gamma(\mathcal{D})$ is bipartite and bidegreed (i.e. each vertex is of degree $r$ or $k$). The spectrum of the graph $\Gamma(\mathcal{D})$ is discussed and determined in \cite{CvDoobSachs}, i.e. $\pm\sqrt{rk}$, $[\pm\sqrt{r-\lambda^{\ast}}]^{v-1}$ and $[0]^{b-v}$. In the following, we are interested in symmetric 2-designs with parameters $(4u^2,2u^2-u,u^2-u)$, which are called Menon designs \cite{Menon1962}. The incidence graphs of symmetric designs are bipartite distance-regular graphs of diameter 3, whose eigenvalues are equal to $\pm r$ and $\pm[\sqrt{r-\lambda^{\ast}}]^{v-1}$.

As announced, below we prove the statement showing the reaching of the bound in the inequality (\ref{bound}).

\begin{proposition}\label{sdistKD}
Let $\Gamma(\mathcal{D})$ be the $n$-vertex, $n$ is even, incidence graph of a Menon 2-design with the parameters $(4u^2,2u^2-u,u^2-u)$. Then
\begin{equation*}
\sigma\left(K_{\frac{n}{2},\frac{n}{2}},\Gamma(\mathcal{D})\right)=\left(\frac{n}{2}\right)^{3/2}+\frac{n}{2}.
\end{equation*}
\end{proposition}

\begin{proof}
According to the assumption of the statement, $n=8u^2$, so by use of (\ref{sdistmenon}), we obtain
\begin{align*}
\sigma\left(K_{4u^2,4u^2},\Gamma(\mathcal{D})\right) = &\, \mathcal{E}(\Gamma(\mathcal{D}))+8u^2-4\lambda_1(\Gamma(\mathcal{D}))\\
                                                     = &\, 2\,(2u^2-u)+2u\,(4u^2-1)+8u^2-4\,(2u^2-u)\\
                                                     = &\, 4u^2+8u^3,
\end{align*}
i.e. $\sigma\left(K_{\frac{n}{2},\frac{n}{2}},\Gamma(\mathcal{D})\right)=\left(\frac{n}{2}\right)^{3/2}+\frac{n}{2}$.
\end{proof}

Let us recall that a \emph{bipartite chain graph} (or a double nested graph) $G=(X,Y)$ is a bipartite graph such that the neighborhoods of the vertices in each set from the bipartition $(X,Y)$ are nested in a chain, i.e. form a chain with respect to set inclusion. In more details, the sets of the bipartition $(X,Y)$ of $G$ are partitioned into $r$ non-empty cells, i.e. $X=\bigcup\limits_{i=1}^{r} X_i$ and $Y=\bigcup\limits_{i=1}^{r} Y_i$, such that all vertices in $X_j$, $j=1,2,\ldots,r$, are joined to all vertices in $\bigcup\limits_{k=1}^{r+1-j} Y_k$. That way, if $N_G(v)$ denotes the set of neighbours of a vertex $v$ in $G$, then for vertices $x^{'}\in X_{j+1}$, $x^{''}\in X_j$, $y^{'}\in Y_{k+1}$ and $y^{''}\in Y_k$, $N_G(x^{'})\subset N_G(x^{''})$ and $N_G(y^{'})\subset N_G(y^{''})$ hold.

For $1\leq p<\frac{n}{2}$, let $B_{n,p}$ denote the connected bipartite graph of order $n$ with the bipartition $(X,Y)$, where $X=\{x_1,x_2,\ldots,x_{\frac{n}{2}}\}$ and $Y=\{y_1,y_2,\ldots,y_{\frac{n}{2}}\}$, in which $x_i$, for $i\leq p$, is adjacent to all of $y_1,y_2,\ldots,y_{\frac{n}{2}}$, and $x_i$, for $i>p$, is adjacent to $y_1,y_2,\ldots,y_p$. It can be noticed that $B_{n,p}$ is a bidegreed bipartite chain graph whose sets from the bipartition are partitioned into $2$ non-empty cells. The graph $B_{12,4}$ is depicted in Figure \ref{chainFig}.

\begin{figure}[h]
%\vspace{-10pt}
\centering
\includegraphics[scale=0.8]{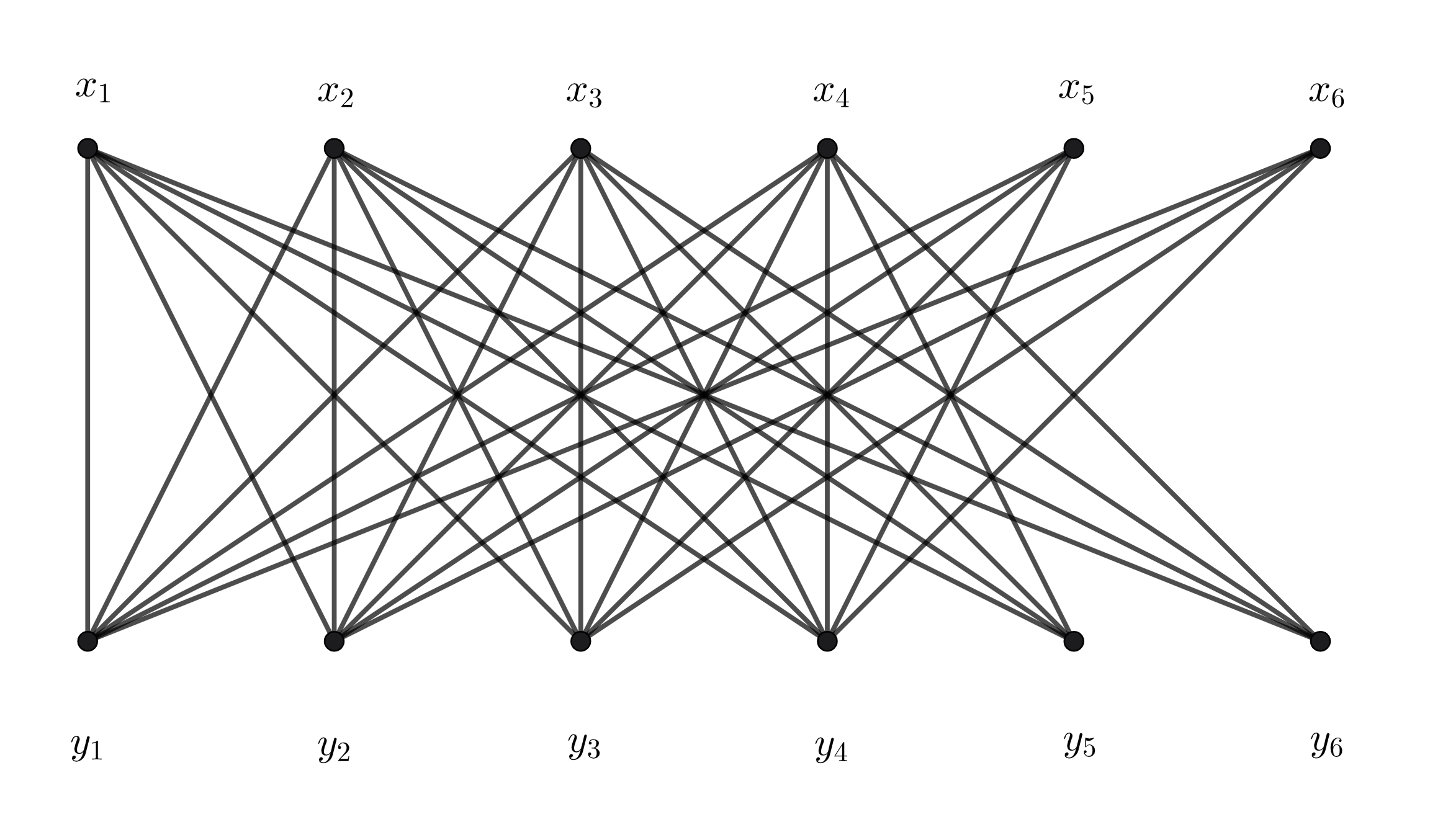}
\vspace{10pt}
\caption{Bipartite graph $B_{12,4}$}
\label{chainFig}
\end{figure}
%\vspace{15pt}

Before we compute the spectrum of the graph $B_{n,p}$, let us recall certain definitions and statements and make some observations which we will need in the proof of the co\-rres\-pon\-ding theorem.

The multiplicity of the eigenvalue zero in the spectrum of $G$, denoted by $\eta=\eta(G)$, is the \emph{nullity} of the graph $G$. If $r(A(G))$ is the rank of the matrix $A=A(G)$, then clearly, $\eta(G)=n-r(A(G))$. The \emph{rank} of a graph $G$ is the rank of its adjacency matrix $A(G)$, denoted by $r(G)$. Then, $\eta(G)=n-r(G)$. It is known \cite{Higgins} that for the bipartite graph $G$ with $n$ vertices and the biadjacency matrix $B$, $\eta(G)=n-2r(B)$ holds, where $r(B)$ is the rank of the matrix $B$. From the definition of the graph $B_{n,p}$ it is obvious that the rank of its biadjacency matrix is equal to $2$, which implies $\eta(B_{n,p})=n-4$, i.e. $B_{n,p}$ has four non-zero eigenvalues.

Let $G$ be a graph with the vertex set $V(G)$. The partition $V(G)=V_1\,\dot{\cup}\,V_2\,\dot{\cup}\cdots\dot{\cup}\,V_k$, where $\dot{\cup}$ stands for the disjoint union, is an \emph{equitable partition} if every vertex in $V_i$ has the same number of neighbours in $V_j$, for all $i,j\in\{1, 2, \ldots, k\}$ (for more details see e.g. \cite{CvRowSim}, p. 83).

The following statement can be found in \cite[Corollary 1.3.13]{CvRowSim}.

\begin{corollary}\label{coll}
Let $A$ be a real symmetric matrix with spectrum $\lambda_1\geq\lambda_2\geq\cdots\geq\lambda_n$. Given a partition $\{1, 2,\ldots, n\}=\Delta_1\,\dot{\cup}\,\Delta_2\,\dot{\cup}\cdots\dot{\cup}\,\Delta_m$ with $|\Delta_i|=n_i>0$, consider the corresponding blocking $A=(A_{ij})$, where $A_{ij}$ is an $n_i\times n_j$ block. Let $e_{ij}$ be the sum of the entries in $A_{ij}$ and set $C=(e_{ij}/n_i)$ (note that $e_{ij}/n_i$ is the average row sum in $A_{ij}$). Then the eigenvalues of $C$ interlace those of $A$.
\end{corollary}

The following theorem can be found in \cite[Theorem 1.3.14]{CvRowSim}.

\begin{theorem}\label{Theorem 1.3.14}
Let $A$ be any matrix partitioned into blocks, as in Corollary \ref{coll}. Let us suppose that the block $A_{ij}$ has constant row sums $c_{ij}$, and let $C=(c_{ij})$. Then the spectrum of $C$ is contained in the spectrum of $A$ (taking into account the multiplicities of the eigenvalues).
\end{theorem}

\begin{lemma}\label{specB_np}
Let $n\geq 4$ be an even integer and $1\leq p<\frac{n}{2}$. The spectrum of the bipartite chain graph $B_{n,p}$ has exactly four non-zero eigenvalues as follows $\lambda_1(B_{n,p})=\frac{1}{2}\left(p+\sqrt{p\left(2n-3p\right)}\right)$, $\lambda_2(B_{n,p})=\frac{1}{2}\left(-p+\sqrt{p\left(2n-3p\right)}\right)$, $\lambda_{n-1}(B_{n,p})=-\lambda_2(B_{n,p})$ and $\lambda_n(B_{n,p})=-\lambda_1(B_{n,p})$. Furthermore, the two largest eigenvalues satisfy the
following identity
\begin{equation*}
\lambda_1(B_{n,p})+\lambda_2(B_{n,p})=\sqrt{p\left(2n-3p\right)}.
\end{equation*} 
\end{lemma}

\begin{proof}
As we previously noted in the text above, $\eta(B_{n,p})=n-4$, which means that the spectrum of $B_{n,p}$ contains eigenvalue zero of the multiplicity $n-4$. 

The partition $V(B_{n,p})=V_1\,\dot{\cup}\,V_2\,\dot{\cup}\,V_3\,\dot{\cup}\,V_4$ of the vertex set $V(B_{n,p})$ of the graph $B_{n,p}$, where $V_1=\{x_1,\ldots,x_p\}$, $V_2=\{x_{p+1},\ldots,x_{\frac{n}{2}}\}$, $V_3=\{y_1,\ldots,y_p\}$ and $V_4=\{y_{p+1},\ldots,y_{\frac{n}{2}}\}$, is an equitable partition with the following quotient matrix
$$Q=\left(
      \begin{array}{cccc}
        0 & 0 & p & \frac{n}{2}-p \\
        0 & 0 & p & 0 \\
        p & \frac{n}{2}-p & 0 & 0 \\
        p & 0 & 0 & 0 \\
      \end{array}
    \right).
$$
The characteristic polynomial of the matrix $Q$ is 
$$q(x)=p^2\left(\frac{n}{2}-p\right)^2+p(p-n)x^2+x^4.$$
Since, according to Theorem \ref{Theorem 1.3.14}, the roots of the polynomial $q(x)$, i.e. the eigenvalues of the matrix $Q$ are also the eigenvalues of the graph $B_{n,p}$, the first part of the statement is proved. 

The second part of the proof can be obtained by direct computation.
\end{proof}

Now, we are ready for the main statement of the section.

\begin{theorem}\label{main_bip}
Let $\Gamma(\mathcal{D})$ be the $n$-vertex incidence graph of a Menon 2-design with the parameters $(4u^2,2u^2-u,u^2-u)$, $u\geq 4$ \footnote{Such designs exist, see, for example, \cite{comb_designs}.}, and let $B_{n,p}$ be the $n$-vertex bipartite chain graph, where $p=\frac{n}{3}$, if $3\,|\,u$, and $p=\frac{n}{3}+\frac{1}{3}$, otherwise. Then 
\begin{equation*}
\sigma(\Gamma(\mathcal{D}),B_{n,p})>\sigma(K_{\frac{n}{2},\frac{n}{2}},\Gamma(\mathcal{D})).
\end{equation*}
\end{theorem}

\begin{proof}
Let us first compute $\sigma(\Gamma(\mathcal{D}),B_{n,p})$. By use of Lemma \ref{bipdistance}, we obtain
\begin{equation}\label{sdistBD}
\sigma(\Gamma(\mathcal{D}),B_{n,p})=\mathcal{E}(\Gamma(\mathcal{D}))+\mathcal{E}(B_{n,p})-4\sum\limits_{i=1}^{2}\min\{\lambda_i(\Gamma(\mathcal{D})),\lambda_i(B_{n,p})\},
\end{equation} 
since $\eta(B_{n,p})=n-4$.

According to the assumptions of the statement, $\frac{n}{3}\leq p\leq \frac{n+1}{3}$, so, since $p(2n-3p)$ is concave in $p$, the following inequality holds
$$p(2n-3p)\geq\frac{n+1}{3}(n-1)=\frac{n^2-1}{3}.$$
Therefore, using Lemma \ref{specB_np} and bearing in mind the values of the eigenvalues of the graph $\Gamma(\mathcal{D})$ mentioned before in the section, we find
\begin{equation*}
\lambda_1(B_{n,p})\geq\frac{n}{6}+\frac{\sqrt{n^2-1}}{2\sqrt{3}}>\frac{n}{4}-\frac{\sqrt{n}}{2\sqrt{2}}=\lambda_1(\Gamma(\mathcal{D})),
\end{equation*}
and  
\begin{equation*}
\lambda_2(B_{n,p})\geq-\frac{n+1}{6}+\frac{\sqrt{n^2-1}}{2\sqrt{3}}\geq\frac{\sqrt{n}}{2\sqrt{2}}=\lambda_2(\Gamma(\mathcal{D})),
\end{equation*}
where both inequalities are true for all $n\geq 12$.

Now, (\ref{sdistBD}) reduces to
\begin{equation*}
\sigma(\Gamma(\mathcal{D}),B_{n,p})=\mathcal{E}(\Gamma(\mathcal{D}))+\mathcal{E}(B_{n,p})-4(\lambda_1(\Gamma(\mathcal{D}))+\lambda_2(\Gamma(\mathcal{D}))),
\end{equation*}
i.e. using Lemma \ref{specB_np} and knowing the values for the eigenvalues of the graph $\Gamma(\mathcal{D})$ mentioned before in the section
\begin{equation}\label{distBD_final}
\sigma(\Gamma(\mathcal{D}),B_{n,p})=\left(\frac{n}{2}\right)^{3/2}-\frac{n}{2}-\frac{2\sqrt{n}}{\sqrt{2}}+2\sqrt{p(2n-3p)}.
\end{equation}

By use of Proposition \ref{sdistKD} and the equation (\ref{distBD_final}), we obtain
\begin{equation*}
\sigma(\Gamma(\mathcal{D}),B_{n,p})-\sigma(K_{\frac{n}{2},\frac{n}{2}},\Gamma(\mathcal{D}))=2\sqrt{p(2n-3p)}-\frac{2\sqrt{n}}{\sqrt{2}}-n>0,
\end{equation*}
which is true, since for $\frac{n}{3}\leq p\leq\frac{n+1}{3}$,
\begin{equation*}
2\sqrt{p(2n-3p)}-\frac{2\sqrt{n}}{\sqrt{2}}-n\geq 2\sqrt{\frac{n^2-1}{3}}-\sqrt{2n}-n>0
\end{equation*}
holds for every $n\geq 84$. This completes the proof.
\end{proof}

Regarding the proof, Theorem \ref{main_bip} holds for $n\geq 84$, but since it is assumed $u\geq 4$, the smallest possible value for $n$ is $128$. Namely, a symmetric design for $u=4$ surely exists (see, for example, \cite{2designs} and references therein), and the spectrum of its incidence graph consists of the following eigenvalues $\pm 28$ and $\pm[4]^{63}$. On the other hand, the eigenvalues of the graph $B_{128,43}$ are approximately $\pm 58.45$, $\pm 15.45$ and $[0]^{124}$, while the eigenvalues of $K_{64,64}$ are $\pm 64$ and $[0]^{126}$. Therefore, we have 
$$579.8\approx\sigma(\Gamma(\mathcal{D}),B_{128,43})>576=\sigma\left(K_{64,64},\Gamma(\mathcal{D})\right).$$
Since, by Theorem \ref{seccB_n} and Proposition \ref{sdistKD}, on the set $\mathcal{B}_n$ of all connected bipartite graphs of even order $n$, 
${\rm secc}_{\mathcal{B}_n}\left(K_{\frac{n}{2},\frac{n}{2}}\right)=\sigma\left(K_{\frac{n}{2},\frac{n}{2}},\Gamma(\mathcal{D})\right)$ holds, we actually have
$${\rm secc}_{\mathcal{B}_{128}}\left(K_{64,64}\right)=\sigma\left(K_{64,64},\Gamma(\mathcal{D})\right)<\sigma(\Gamma(\mathcal{D}),B_{128,43}),$$
which disproves Conjecture \ref{con_bip}.  

The conducted exhaustive computer search showed that Conjecture \ref{con_bip} holds for all connected bipartite graphs whose order is not greater than $12$, and that the smallest order $n$ of connected bipartite graphs where Conjecture \ref{con_bip} fails is $n=72$. In this case, the spectral eccentricity ${\rm secc}_{\mathcal{B}_{72}} (K_{36,36})$ of the complete bipartite graph $K_{36,36}$ is strictly smaller than the spectral distance between a Menon design with parameters $(36,15,6)$ and a connected bipartite graph whose structure is different than those of the considered bipartite chain graphs.

\section{Concluding remarks}

The paper demonstrates that both conjectures of Stani\' c fail. The exact values of $\mathrm{sdiam}(\mathcal{T}_n)$ and $\mathrm{sdiam}(\mathcal{B}_n)$, as well as the corresponding extremal pairs of graphs remain unknown. By Theorem \ref{main1}, the pair of graphs $P_n$ and $K_{1,n-1}$ is not the extremal one on the set of all $n$-vertex trees, for every $n\geq 95$, while according to Theorem \ref{main_bip}, the pair of graphs $K_{\frac{n}{2},\frac{n}{2}}$ and $\Gamma(\mathcal{D})$ is not extremal on the set of bipartite graphs of order $n$ when $n=8u^2$, $u\geq 4$. It would be interesting to determine the extremal pairs of graphs on the set of all trees, as well as to consider graphs of odd order regarding Conjecture \ref{con_bip}.

\bigskip

{\sloppy\bigskip\noindent\textbf{Acknowledgments.}~~The author Škrekovski R. acknowledges the partial support by ARIS program P1-0383, bilateral Slovenian-Croatian project BI-HR/25-27-004 and the annual work program of Rudolfovo, while the author Sedlar J. acknowledges the support by Project KK.\allowbreak01.\allowbreak 1.\allowbreak1.\allowbreak02.\allowbreak0027 co-financed by the European Regional Development Fund, by Croatian Ministry of Science, Education and Youth through the bilateral Croatian-Slovenian project 2025-26, and by the NextGeneration EU foundation via IP-UNIST-17 (GEORAZ). }

\end{document}